\documentclass[a4paper,12pt]{article}
\usepackage[centertags]{amsmath}
\usepackage{amsfonts}
\usepackage{amssymb}
\usepackage{amsthm}
\usepackage{dsfont}
\usepackage{tikz, subfigure}
\usepackage{parskip}

\newtheorem{theorem}{Theorem}[section]
\newtheorem{lemma}{Lemma}[section]

\newtheorem{defn}{Definition}[section]

\begin{document}

\title{Digit Frequencies and Bernoulli Convolutions}
\author{Tom Kempton}
\maketitle

\begin{abstract}\noindent It is well known that the Bernoulli convolution $\nu_{\beta}$ associated to the golden mean has Hausdorff dimension less than $1$, i.e. that there exists a set $A$ with $\nu_{\beta}(A)=1$ and $dim_H(A)<1$. We construct such a set $A$ explicitly and discuss how our approach might be generalised to prove the singularity of other Bernoulli convolutions\end{abstract}
\section{Introduction}
Bernoulli convolutions are a simple class of fractal measures. Given $\beta\in(1,2)$, we let $\pi:\{0,1\}^{\mathbb N}\to I_{\beta}:=\left[0,\frac{1}{\beta-1}\right]$ be given by
\[
\pi(\underline a)=\sum_{i=1}^{\infty}a_i\beta^{-i}.
\]
We let $m$ be the $\left(\frac{1}{2},\frac{1}{2}\right)$ Bernoulli measure on $\{0,1\}^{\mathbb N}$ and define the Bernoulli convolution $\nu_{\beta}$\footnote{Unfortunately, researchers in the closely related fields of Bernoulli convolutions and $\beta$-transformations tend to use different notation. We adopt the language of $\beta$-transformations, Bernoulli convolutions are often denoted $\nu_{\lambda}$ where $\lambda$ corresponds to $\frac{1}{\beta}$ in our language.} by
\[
\nu_{\beta}=m\circ\pi^{-1}.
\]

It is a major open problem to determine the values of $\beta$ for which $\nu_{\beta}$ is absolutely continuous. Erd\H{o}s proved in \cite{ErdosPisot} that if $\beta\in(1,2)$ is a Pisot number then the Bernoulli convolution $\nu_{\beta}$ is singular, while in \cite{GarsiaAC} Garsia provided a countable number of values of $\beta$ for which $\nu_{\beta}$ is absolutely continuous. Solomyak proved in \cite{SolomyakAC} that for almost every $\beta\in(1,2)$ the corresponding Bernoulli convolution is absolutely continuous, yet for typical given $\beta$ the problem of determining whether $\nu_{\beta}$ is absolutely continuous remains open. 

The case that $\beta$ is equal to the golden mean is of particular interest, with extensive investigations into questions relating to entropy, Hausdorff dimension, Gibbs properties and multifractal analysis of the corresponding Bernoulli convolution, see for example \cite{AlexanderZagier,HuTian,LedrappierPorzio1,LedrappierPorzio2,OlivierSidorovThomas,SidorovVershik}. Given that $\nu_{\beta}$ is totally singular, it seems that one of the most fundamental question associated to $\nu_{\beta}$ is to understand where it is supported, can one describe a set of zero Lebesgue measure but with full $\nu_{\beta}$-measure? 

In this note we answer the above question by giving an example of a set $A$ of Hausdorff dimension less than one but with $\nu_{\beta}(A)=1$. Our method uses digit frequencies for the greedy $\beta$-transformation and provides a new elementary proof of the singularity of the golden mean Bernoulli convolution. We also hope that, by describing the manner of the singularity of $\nu_{\beta}$ for $\beta$ equal to the golden mean, we may be able to shed some light on the question of which other Bernoulli convolutions are singular. We comment on the feasibility of extending our method to other Bernoulli convolutions in the final section.

\subsection{The Greedy $\beta$-transformation}
Let $\beta=\frac{1+\sqrt{5}}{2}$ be the golden mean. In \cite{Renyi}, Renyi introduced the greedy $\beta$-transformation $T:I_{\beta}\to I_{\beta}$ defined by
\[
T(x):=\left\lbrace\begin{array}{cc}
                  \beta x & x\in[0,\frac{1}{\beta})\\
\beta x-1& x\in[\frac{1}{\beta},\frac{1}{\beta-1}]
                  \end{array}\right. .
\]
The map $T$ has $[0,1]$ as an attractor, and $T(x)|_{[0,1]}=\beta x$ (mod 1). In \cite{Parry} it was proved that T preserves the absolutely continuous probability measure $\mu$ with density given by
\[
\rho(x)=C\sum_{i=0}^{\infty} \frac{1}{\beta^i}\chi_{[0,T^i(1)]}(x)=\left\lbrace\begin{array}{cc}
                                                              \dfrac{1}{\frac{1}{\beta}+\frac{1}{\beta^3}}&x\in[0,\frac{1}{\beta})\\
\dfrac{1}{1+\frac{1}{\beta^2}}&x\in[\frac{1}{\beta},1]\\ 
0 & x\in[1,\frac{1}{\beta-1}]\end{array}\right. ,
\]
where $C$ is a normalising constant. Given $x\in I_{\beta}$ we generate a sequence $(x_n)$ known as the greedy $\beta$-expansion of $x$ by iterating $T$. Letting $x_n=0$ when $T^n(x)=\beta (T^{n-1}x)$ and $x_n=1$ when $T^n(x)=\beta (T^{n-1}x)-1$, we have that the sequence $(x_n)$ satisfies
\[
\sum_{n=1}^{\infty}x_n\beta^{-n}=x.
\]

We let $\Sigma:=\{0,1\}^{\mathbb N}$ and let
\[
X_{\beta}:=\overline{\left\{(x_n)\in\Sigma:(x_n)\text{ is a greedy }\beta\text{-expansion of some }x\in\left[0,\frac{1}{\beta-1}\right]\right\}}, 
\]
where $\overline X$ denotes the closure of $X$. We define the left shift $\sigma:\Sigma\to\Sigma$ by $\sigma(x_n)_{n=1}^{\infty}=(x_{n+1})_{n=1}^{\infty}$. We further define the lexicographical ordering $\prec$ on $\Sigma$ by declaring that $(x_n)\prec (y_n)$ if and only if $x_1<y_1$ or if there exists $n\in\mathbb N$ with $x_1\cdots x_{n-1}=y_1\cdots y_{n-1}$ and $x_n<y_n$. 

In \cite{Parry}, Parry characterised the set of $\beta$-expansions of $x\in I_{\beta}$ in terms of the orbit of the point $1$ under $T_{\beta}$. For $\beta$ equal to the golden mean we have that any sequence $(x_n)$ without two consecutive occurences of digit $1$ is in the closure of the set of greedy $\beta$-expansions of points $x\in[0,1)$. Points $x\in[1,\frac{1}{\beta-1})$ have greedy expansions which start with $m$ $1$s for some $m\geq 2$, but have that $\sigma^k(x_n)\prec(b_n)$ for each $k\geq m-1$.

Applying the ergodic theorem to $(I_{\beta},\mu,T)$ with the characteristic function on $[\frac{1}{\beta},\beta]$ gives the following theorem.
\begin{theorem}[Parry]
Let $\beta=\frac{1+\sqrt{5}}{2}$. For Lebesgue almost every $x\in I_{\beta}$ the frequency of the digit $1$ in the greedy $\beta$-expansion of $x$ is given by
\[
\alpha(1):=\mu\left[\frac{1}{\beta},1\right]=\frac{1}{\beta^2+1}\approx 0.27639 .
\]
\end{theorem}

Indeed, it follows from standard techniques in multifractal analysis that, letting
\[
A_{\gamma}=\{x\in I_{\beta}: \text{ the greedy $\beta$-expansion of $x$ has digit $1$ with frequency } \gamma\},
\]
we have that $\dim_H(A_\gamma)<1$ for each $\gamma\neq\alpha(1)$\footnote{To prove this, first consider the corresponding symbolic set
\[
B_{\gamma}=\{(x_n)\in X_{\beta}: \lim_{n\to\infty}\frac{1}{n}\sum_{k=1}^n x_n=\gamma\}=\pi^{-1}(A_{\gamma}).
\]
For $\gamma\neq \alpha(1)$ we have that each measure supported on $B_{\gamma}$ has metric entropy strictly less than $\log(\beta)$, since $X_{\beta}$ has a unique measure of maximal entropy with entropy $\log(\beta)$ which is supported on $B_{\alpha(1)}$. Hence by the conditional variational principle for Markov maps (see \cite{TakensVerbitskiy,PfisterSullivan}), the topological entropy of $B_{\gamma}$ is strictly less than $\log(\beta)$. Then dividing by the lyapunov exponent $\log(\beta)$, we see that the Hausdorff dimension of $B_{\gamma}$ is strictly less than one. Finally we have that the Hausdorff dimension of $A_{\gamma}$ is less than one because Hausdorff dimension does not increase under the projection $\pi$.}.


Simple analysis reveals that $\nu_{\beta}$-almost every $x\in[0,1]$ has $\beta$-expansion in which the digit $1$ occurs with frequency different from $\alpha(1)$.

\begin{theorem}\label{MainTheorem}
Let $\beta=\frac{1+\sqrt{5}}{2}$. For $\nu_{\beta}$-almost every $x\in I_{\beta}$, the frequency of the digit $1$ in the greedy $\beta$-expansion of $x$ is equal to $\frac{5}{18}=0.2\dot 7 >\alpha(1)$.
\end{theorem}
We prove this theorem in the next section. An immediate corollary is that $\nu_{\beta}$ is singular with Hausdorff dimension
\[
\dim_H(\nu_{\beta})\leq\dim_H(A_{\frac{5}{18}})<1.
\]

\section{Normalising $\beta$-expansions}
We define a map $P:\Sigma\to X_{\beta}$ by $P(a_n)=(x_n)$ where $x_n$ is the greedy expansion of $\sum_{i=1}^{\infty}a_i\beta^{-i}$. This process is called normalising and has been studied, for example, in \cite{DanielFrougny,SidorovVershik}.

\begin{lemma}[Sidorov, Vershik \cite{SidorovVershik}]
Given $(a_n)\in \Sigma$, the sequence $P(a_n)$ can be found as the limit of the sequences obtained by iterating the following algorithm. First look for the first occurence in $(a_n)$ of the word $011$. If such a word occurs, replace it with $100$ and then repeat from the start, looking for the first occurence of the word $011$ in our new sequence.
\end{lemma}

\begin{proof}
Because the golden mean satisfies the equation
\[
\beta^{-n}=\beta^{-(n+1)}+\beta^{-(n+2)}
\]
for any $n\in\mathbb N$ we see that the above algorithm produces another $\beta$-expansion of the same point. Furthermore, the sequence $P(a_n)$ will have no two consecutive occurrences of digit $1$, except that it may start with word $1^m$, and thus $\sigma^kP(a_n)\prec (b_n)$ for all $k\geq m$ giving that $P(a_n)$ is a greedy sequence. 
\end{proof}

We extend the definition of $P$ to almost all two-sided sequences $(a_n)\in\{0,1\}^{\mathbb Z}$ as follows. First we define $P$ on finite words $a_1\cdots a_n$ as in the above algorithm. For a sequence $(a_n)_{n=1}^{\infty}$ for which there exist infinitely many $n\in\mathbb N$ for which $a_n=a_{n+1}=0$, we can rewrite $(a_n)=w_1w_2\cdots$ where $w_i$ is a finite word which starts with $00$ for each $i\geq 2$. Then we observe that $P(a_n)=P(w_1)P(w_2)\cdots$, see \cite{SidorovVershik} for a proof. Finally, if $(a_n)_{n=-\infty}^{\infty}$ is a two sided sequence for which there exist infinitely many positive and negative $n$ with $a_n=a_{n+1}=0$, we can again write $(a_n)$ as a concatenation of words $w_i$ that start with $00$ and define $P(a_n)$ to be the concatenation of the normalised words $P(w_i)$.

We also observe that if one has a two sided sequence $(x_n)_{n=-\infty}^{\infty}$ and one defines $y^+=P((x_n)_{n=1}^{\infty})$ and $y^-=P((x_n)_{n=-\infty}^{0})$ then $P((x_n)_{n=-\infty}^{\infty})=P(y_-y^+)$.

Our strategy goes as follows. In order to find the expected asymptotic frequency of the digit $1$ in the sequence $P(x^+)$ for a randomly chosen $x^+=(x_n)_{n=1}^{\infty}$ we choose a random `past' $x^-=(x_n)_{n=-\infty}^{0}$, where the negative digits are chosen independently from $\{0,1\}$ with probability $\frac{1}{2}$, and instead normalise the two sided sequence $x^-x^+$. The two sided sequence $P(x^-x^+)$ agrees with $P(x^+)$ for all positive $n$ for which there exists $0<m<n$ with $x_m=x_{m+1}=0$, and thus has the same limiting frequency of the digit $1$ in the positive coordinates.

Since the probability that the $n$th term of $P(x^-x^+)$ is equal to $1$ is independent of $n$, we just need to calculate the probability that $(P(x^-x^+))_0=1$ and then use the ergodic theorem together with the shift map.

Now since $(\Sigma,m,\sigma)$ is ergodic, the ergodic theorem using the characteristic function on the set of sequences with $(P(x^-x^+))_0=1$ gives that the limiting frequency of the digit $1$ in sequences $P(a_n)$ is equal to
\[
m\{(a_n)\in\Sigma: (P(a_n))_0=1\}.
\]

In the next section we calculate the measure of this set, we introduce the shorthand
\[
\mathbb P(x_n=i):=m( a\in\Sigma:(P(a))_n=i\}
\]
and refer to this as the probability that $x_n$ is equal to $i$.

\subsection{Calculations}
In this section we make some simple computations which prove Theorem \ref{MainTheorem}.
\begin{lemma}
Let $P(x^+)=(y_n)_{n=1}^{\infty}$. Then the probability that $y_1=1$ is equal to $\frac{2}{3}$,and the probability that $y_1=y_2=1$ is equal to $\frac{1}{3}$.
\end{lemma}

\begin{proof} The cylinder sets $A_k=[(01)^k1]$ and $B_k=[(01)^k00]$ for $k\geq 0$ partition $\{0,1\}^{\mathbb N}$. Then using that $P((01)^k1)=10^{2k}$ and $P((01)^k00)=(01)^k00$ we see that $y_1=1$ if and only if $x^+\in A_k$ for some $k\geq 0$. Thus the probability that  $y_1=1$ is given by
\[
\mathbb P(y_1=1)=\sum_{k=0}^{\infty}\frac{1}{2}.\frac{1}{4^k}=\frac{1}{2}\left(\frac{1}{1-\frac{1}{4}}\right)=\frac{2}{3}
\]
Similarly, we see that $y_1=y_2=1$ if and only if $x^+$ belongs to the cylinder set $[1(01)^k1]$ for some $k\geq 0$, which happens with probability
\[
\mathbb P(y_1=y_2=1)=\sum_{k=0}^{\infty}\frac{1}{4}.\frac{1}{4^k}=\frac{1}{4}\left(\frac{1}{1-\frac{1}{4}}\right)=\frac{1}{3}
\]
\end{proof}

\begin{lemma}
Let $P(x^-)=(y_n)_{n=-\infty}^{0}$. Then the probability that $y_0=1$ is equal to $\frac{1}{3}$, and the probability that $y_{-1}=y_0=0$ is equal to $\frac{1}{2}$.
\end{lemma}

\begin{proof} The cylinder sets $A_k:=_{-2k}[01^{2k}]_0$ and $B_k:=_{-2k-1}[01^{2k+1}]_0$ for $k\geq 0$ partition the set $\{(x_n)_{n=-\infty}^0:x_i\in\{0,1\}\}$. We use the identities $P(01^{2k})=(10)^k0$ and $P(01^{2k+1})=(10)^k01$. Then we see that $y_0=1$ if and only if $x^{-}\in B_k$ for some $k\geq 0$. This happens with probability
\[
\mathbb P(y_0=1)=\sum_{k=0}^{\infty}\frac{1}{4}\left(\frac{1}{4}\right)^k=\frac{1}{3}
\]
Furthermore we see that $y_0=y_{-1}=0$ if and only if $x_{-2k}\cdots x_0=01^{2k}$ for some $k\geq 1$ or $x_{-2k-1}\cdots x_0=01^{2k}0$ for some $k\geq 0$. This happens with probability
\[
\mathbb P(y_{-1}=y_0=0)=\left(\sum_{k=1}^{\infty}\frac{1}{2}\left(\frac{1}{4}\right)^k\right)+\left(\sum_{k=0}^{\infty}\left(\frac{1}{4}\right)^{k+1}\right)=\frac{1}{6}+\frac{1}{3}=\frac{1}{2}
\]
\end{proof}

Finally, we study $P(y^-y^+)$.
\begin{lemma}
Letting $(x_n)_{n=-\infty}^{\infty}=P(y^-y^+)$ we have that $x_0=1$ if and only if one of the following two conditions holds.

{\bf Case 1:} $y_0=1$, $y_{1}=0$.

{\bf Case 2:} $y_{-1}=y_0=0$, $y_{1}=y_{2}=1$.

Thus the probability that $x_0=1$ is equal to $\frac{5}{18}$.
\end{lemma}

\begin{proof}
$P(y^-y^+)=(x_n)_{n=1}^{\infty}$. Since $y^-$ and $y^+$ are already normalised, we see that if $y_1=0$ then the concatenation $y^-y^+$ does not contain consecutive occurences of the digit $1$ and hence is normalised, giving $x_0=y_0$. Thus if $y_1=0$ then $x_0=1$ if and only if $y_0=1$.

Now suppose that $y_1=1$. If $y_0=1$ then, since $y^-$ is normalised we must have that $y_{-1}=0$. Then when we normalise $y^-y^+$ the $011$ in position $y_{-1}y_0y_{1}$ is replaced by $100$, giving that $x_0=0$.

We suppose that $y_1=1$ and $y_0=0$. If $y_2=0$ then $y^-y^+$ is normalised and $x_0=0$. Thus we assume that $y_2=1$. Then the first action of $P$ is to turn the $011$ in position $y_0y_1y_2$ into $100$. If $y_{-1}=1$ then subsequent substitutions of $011$ with $100$ will leave $x_n=0$. However if $y_{-1}=0$ then this sequence is normalised. 

Adding the relevant probabilities computed earlier, we see that
\begin{eqnarray*}
\mathbb P(x_0=1)&=&\mathbb P(y_0=1,y_{1}=0)+\mathbb P(y_{-1}=y_0=0,y_{1}=y_{2}=1)\\
&=& \frac{1}{3}.\frac{1}{3}+\frac{1}{2}.\frac{1}{3}\\
&=&\frac{5}{18}
\end{eqnarray*}\end{proof}
Then by the ergodic theorem, we observe that for $m-$almost every sequence $(a_n)\in\Sigma$, the sequence $P(a_n)$ has digit $1$ with limiting frequency $\frac{5}{18}$. Since $\nu_{\beta}=m\circ \pi^{-1}$ it follows that $\nu_{\beta}$ is supported on the set $A_{\frac{5}{18}}$.

\section{Further Questions and Comments}

It is straightforward to generalise our method to the so called Multinacci numbers, that is $\beta$ satisfying the equation
\[
\beta^n-\beta^{n-1}-\beta^{n-2}\cdots -1=0
\]
for some $n\geq 2$. 

{\bf Question 1:} Can one generalise the method in this article to all Pisot numbers $\beta\in(1,2)$? How about all algebraic $\beta\in(1,2)$?

We have not been able to extend our method beyond the multinacci case, but in this final section we give some ideas about how we believe such a generalisation might go for certain algebraic $\beta$.

We let $P_{\beta}$ denote the normalisation operator associated to $\beta$ which maps a sequence $(a_n)\in\Sigma$ to the corresponding greedy sequence in $X_{\beta}$. 

\begin{defn} We say that $P_{\beta}$ has the finite word property if there is a set of finite words $\Omega=\{\omega_i, i\in\mathbb N\}$ such that the cylinder sets $\{[\omega_i]:i\in\mathbb N\}$ partition $\Sigma$ and such that the following two properties hold.
\begin{enumerate}
 \item For a sequence $(a_n)\in\Sigma$ we can calculate $P_{\beta}(a_n)$ by writing $(a_n)$ as a concatenation of words from $\Omega$ and applying $P_{\beta}$ to each of these words, giving $P_{\beta}(\omega_{a_1}\omega_{a_2}\cdots)=P_{\beta}(\omega_{a_1})P_{\beta}(\omega_{a_2})\cdots$.
 \item The expected length of $\omega_i$, equal to $\sum_{i=1}^{\infty}|\omega_i|2^{-|\omega_i|}$, is finite.
\end{enumerate}
\end{defn}
For $\beta$ equal to the golden mean we could choose $\Omega$ to be the set of words which end with two consecutive zeros and which have no intermediate occurrence of two consecutive zeros.

{\bf Question 2:} Do there exist non-Pisot values of $\beta$ with the finite word property?

Given a value of $\beta$ for which $P_{\beta}$ has the finite words property one can write a computer program to calculate the frequency of $1$s in $P_{\beta}(a_n)$ for typical $(a_n)\in\Sigma$. This value is equal to
\[
\sum_{i=1}^{\infty} \frac{|\omega_i|2^{-|\omega_i|}\times \text{ frequency of digit $1$ in $\omega_i$}}{\sum_{i=1}^{\infty}|\omega_i|2^{-|\omega_i|}},
\]
where the sum converges because of the finiteness of $\sum_{i=1}^{\infty}|\omega_i|2^{-|\omega_i|}$.

The so called Garsia separation property, which is satisfied when $\beta$ is a Pisot number, implies that $P_{\beta}$ has the finite words property. Thus, given any Pisot number, one could approximate the corresponding digit frequencies. We do not know of any non-computational way of calculating digit frequencies for general Pisot numbers $\beta$. 

It seems quite possible that other algebraic values of $\beta$ may also have the finite words property and would be very interesting to attempt to calculate digit frequencies in such cases.

\section*{Acknowledgements}
Many thanks to Karma Dajani and Vaughn Climenhaga for many interesting and helpful discussions. This work was supported by the Dutch Organisation for Scientific Research (NWO) grant number 613.001.022. 

\bibliographystyle{plain} 
\bibliography{betaref.bib}

\end{document}